\documentclass[11pt]{amsart}
\usepackage{amsmath,amsfonts,amsthm,amssymb,amscd}
\def\classification#1{\def\@class{#1}}
\classification{\null}

\DeclareMathAlphabet{\mathscr}{OT1}{rsfs}{n}{it}

\newcommand{\R}{{\mathbb R}}

\newcommand{\C}{\mathbb{C}}

\newtheorem{lemma}{Lemma}

\newtheorem{theorem}{Theorem}
\theoremstyle{remark}

\title{On new sum-product type estimates}

\author{Misha Rudnev}
\address{Misha Rudnev, Department of Mathematics, University of Bristol, Bristol BS8 1TW, United Kingdom} 
\email{m.rudnev@bristol.ac.uk}

\subjclass[2000]{68R05, 11B75}
\begin{document}
\begin{abstract} New lower bounds involving sum, difference, product, and ratio sets for  $A\subset \C$ are given.
\end{abstract}

\maketitle

\section{Introduction} Erd\H os and Szemer\'edi (\cite{ES})  conjectured  that if $A$ is a set of integers, then
$$
|A+A|+|A\cdot A|\gg |A|^{2-o(1)},
$$
where
$$
A+A = \{a_1+a_2:\,a_{1,2}\in A\}
$$
is called the sum set of $A$, the product $A\cdot A$, difference $A-A$, and ratio $A:A$ sets being similarly defined. (In the latter case one should not divide by zero.)  The notations $\ll,\;\gg,\;\approx$ are being used throughout to suppress absolute constants in inequalities, the symbol
$o(1)$ in exponents absorbs logarithmic factors in the asymptotic parameter $|A|$, the cardinality of $A$.

Variations of the Erd\H os-Szemer\'edi question consider the set $A$ living in other rings or fields, as well as replacing, e.g., the sum set with the difference set $A-A.$ The conjecture is far from being
settled, and therefore partial current ``word records" on it vary with such variations of the input.

The best result for reals, for instance, is due to Solymosi (\cite{So1}), claiming
\begin{equation}\label{wr}
|A+A|+|A\cdot A|\gg |A|^{1+\frac{1}{3} - o(1)},
\end{equation}
and would include the endpoint exponent $\frac{4}{3}$ if $A\cdot A$ were replaced by $A:A$.

However, the construction is specific for reals, and does not appear either to extend easily to the case $A\subset \C$, or to allow for replacing the sum set $A+A$ with the difference set $A-A$. So,
if $A\in \C$ or if $A+A$ for reals gets replaced by $A-A$, the best known result comes from an older paper of Solymosi (\cite{So}), claiming
\begin{equation}\label{wr1}
|A\pm A|+|A\cdot A|\gg |A|^{1+\frac{3}{11} - o(1)},
\end{equation}
also without the $o(1)$ term if $A\cdot A$ gets replaced by $A:A$.

It may be worth mentioning that the result (\ref{wr1}) is based on the -- rather sophisticated --  Szemer\'edi-Trotter theorem, while (\ref{wr}) is not, using just elementary order properties of positive reals, expressed in the fact that if $\frac{a}{b}< \frac{c}{d}$, then $\frac{a+c}{b+d}$ falls in between. This note is almost entirely based on the Szemer\'edi-Trotter theorem and succeeds in slightly improving on (\ref{wr1}), yet not enough to beat (\ref{wr}).
\begin{theorem}\label{maint}
For any $A\subset \C$, with two or more elements one has
\begin{equation}\begin{aligned}
|A-A| + |A:A| &\gg & |A|^{1+\frac{9}{31} - o(1)}, \\
|A+A| + |A:A| &\gg & |A|^{1+\frac{15}{53} - o(1)},\\
|A-A| + |A\cdot A| &\gg & |A|^{1+\frac{11}{39} - o(1)},\\
|A+A| + |A\cdot A| &\gg & |A|^{1+\frac{19}{69} - o(1)}.
\end{aligned}
\label{res}\end{equation}
\end{theorem}

\section{Lemmata}
The main tool behind the above estimates is the Szemer\'edi-Trotter incidence theorem. For a set $P$ of points and a set of $L$ straight lines in a plane let
$$
I(P,L)=\{(p,l)\in P\times L:\;p\in l\}
$$
be the set of incidences.

\begin{theorem}[Szemer\'edi and Trotter \cite{ST}] The maximum number of incidences in $\R^2$ is bounded as follows:
\begin{equation}\label{STe}
|I(P,L)|\ll (|P||L|)^{\frac{2}{3}} + |P| + |L|.
\end{equation}
In particular, if $P_t$ (or $L_t$) denote the sets of points (or lines) incident to at least $t\geq 1$ lines (or points) of $L$ (or $P$), then
\begin{equation}\label{work}
\begin{aligned}
|P_t|&\ll & \frac{|L|^2}{t^3} + \frac{|L|}{t},\\
|L_t|&\ll & \frac{|P|^2}{t^3} + \frac{|P|}{t}.
\end{aligned}
\end{equation}
\end{theorem}
Let us note that the linear in $|P|,|L|$ terms in the estimates (\ref{STe}, \ref{work}) are essentially trivial and usually of no interest in the sense of being dominated by the non-linear ones, whenever these estimates are being used. This is also the case in this paper.

The Szemer\'edi-Trotter theorem is also true in the plane
over $\C$. This was proved by T\'oth (\cite{T}). Recently there has been a new proof by Solymosi and Tao  (\cite{SoT}) which had to sacrifice the endpoint in the exponent for the sake of elegance of the method.

One can easily develop a weighted version of the Szemer\'edi-Trotter theorem, see Iosevich et al. (\cite{IKRT}). Suppose each line $l\in L$ has been assigned a weight $m(l)\geq 1$. The number of weighted incidences
$i_m(P,L)$ is obtained by summing over the set $I(P,L)$, with each pair $(p,l)$ being counted $m(l)$ times. Suppose, the total weight of all lines is $W$ and the maximum weight per line is $\bar m$. It is argued in \cite{IKRT} that the worst possible case for the weighted incidence estimate is the uniform one, when there are $\frac{W}{\bar m}$ lines of equal weight $\bar m$, hence the following theorem.

\begin{theorem} \label{STw} The maximum number of weighted incidences between a point set $P$ and a set of lines with the total weight $W$ and  maximum weight per line $\bar m$ is
\begin{equation}\label{STew}
i_m(P,L)\ll \bar m^{\frac{1}{3}} (|P|W)^{\frac{2}{3}} + \bar m |P| + W.
\end{equation}
\end{theorem}

\medskip
This paper uses in its core the same geometric construction as Solymosi (\cite{So}) did, which yielded the exponent $\frac{14}{11}$, with some more detailed analysis of the incidences involved by dealing with the weighted case.
Yet the improvements over (\ref{wr1}) are due to combining this construction with a recent purely additive-combinatorial observation by Shkredov and Schoen (\cite{SS0}, Lemma 3.1) which allowed for a series of the latest state-of-the art
improvements in incremental progress towards a number of open questions in field combinatorics in \cite{SS0}, \cite{SV}, \cite{SS}.

\medskip
The above-mentioned additive-combinatorial observation is the content of the following Lemma \ref{twothr}, quoting which requires some notation used in the sequel. Throughout the rest of this section $A,B$ denote any sets in an Abelian group. The following energy notations $E$, when applied in a field will bifurcate into $E$ and $E_*$, respectively, relative to the addition and multiplication operations.

For any $d\in A-A$, set
\begin{equation}
A_d=\{a\in A: \,a+d\in A\}.\label{ad}
\end{equation}
Denote
$$
E(A,B)=|\{(a_1,a_2,b_1,b_2)\in A\times A\times B\times B:\;a_1-a_2 = b_1-b_2\}|,
$$
referred to as the ``additive energy'' of $A,B$.
By the Cauchy-Schwarz inequality, rearranging the terms in the above definition of $E(A,B)$, one has
\begin{equation}\label{esc}
E(A,B)|A\pm B|\geq |A|^2|B|^2.\end{equation}
Indeed, if $d$ is an element of $A - B$ or $s$ is an element of $A+B$, and $n(d),\,n(s)$ are the number of realisations of $d$ and $s$, respectively as a difference or sum of a pair of elements from $A\times B$,  (\ref{esc}) follows from the fact that
\begin{equation}
E(A,B) = \sum_{d\in A-B} n^2(d) = \sum_{s\in A+B} n^2(s).
\label{epm}\end{equation}

Also, $E(A,A)=E(A)$ is referred to as the ``energy of $A$''. In this case note that according to the notation (\ref{ad}), $n(d)=|A_d|$.

Moreover, the following inequality will be useful. If
\begin{equation}\label{dplus}
D^+ =\left\{d\in A-A:\,n(d)\geq \frac{1}{2}\frac{|A|^2}{|A+A|}\right\},
\end{equation}
then ``the energy supported on $D^+$'', namely the left-hand side of the next inequality, satisfies
\begin{equation}
\sum_{d\in D^+} n^2(d) \gg \frac{|A|^4}{|A+ A|}.\label{needed}\end{equation}
Indeed, by (\ref{esc}) with $A=B$, the energy supported on the complement of $D^+$ is trivially bounded from above by the right-hand side of (\ref{needed}).

Also useful will be the ``cubic energy'' of $A$, which is
\begin{equation}\label{threee}
E_3(A) = |\{(a_1,\ldots,a_6)\in A\times\ldots\times A:\;a_1-a_2=a_3-a_4=a_5-a_6\}|.
\end{equation}
This definition implies that (\cite{SS}, Lemma 2)
\begin{equation}
\label{twothree} E_3(A) = \sum_{d\in A-A} E(A,A_d).
\end{equation}
To see this, fix any $d=a_1-a_2$ in (\ref{threee}) and observe that if one is to count every representation $d=a_3-a_4$ as many times as $a_3,a_4\in A_d$ for some $d$, this will happen exactly $n(d)$ times, for different $d=a_5-a_3=a_6-a_4$ in (\ref{threee}).

The following statement is the content of Corollary 3 in \cite{SS}, with trivial variations.
\begin{lemma}\label{twothr} One has the following identities, for any $D'\in A-A$:
\begin{equation}\begin{aligned}
\label{th} \sum_{d\in D'} |A_d||A\pm A_d| &\geq &\frac{|A|^2\left( \sum_{d\in D'}  |A_d|^{\frac{3}{2}}\right)^2 }{E_3(A)},\\ \hfill \\
\sum_{d\in D'} |A_d|^2|A\pm A_d| &\geq &\frac{|A|^2\left( \sum_{d\in D'}  |A_d|^2\right)^2 }{E_3(A)}.
\end{aligned}
\end{equation}
\end{lemma}

\begin{proof} To verify the first inequality of (\ref{thth}) observe that  by the Cauchy-Schwarz inequality, for each $d$:
$$ \sqrt{|A_d||A\pm A_d|} \sqrt{E(A,A_d)} \geq |A||A_d|^{\frac{3}{2}}.
$$
Summing over $d\in D'$ and applying once again the Cauchy-Schwartz inequality to the left-hand side yields
$$
\sqrt{ \sum_{d\in D'} |A_d||A\pm A_d| } \sqrt{ \sum_{d\in D'} E(A,A_d)}  \geq |A|\sum_{d\in D'}|A_d|^{\frac{3}{2}}.
$$
squaring both sides and using (\ref{twothree}) does the job. Verifying the second inequality in (\ref{th}) requires merely a straightforward modification of the above.\end{proof}

\medskip
The identities (\ref{th}) suggest that the cubic energy estimate from above can be  quite useful, since $A-A_d \subseteq (A-A)\cap (A-A-d)$, as well as $A+A_d\subset (A+A)\cap (A+A+d)$. The latter observation, which \cite{SS} credits to Katz and Koester
(\cite{KK}), allows for the following interpretation of (\ref{th}), producing lower bounds for $E(A,A-A)$ and $E(A,A+A)$. Indeed, the left-hand side in (\ref{twothr}) itself is a lower bound for both $E(A,A\pm A)$. However, in order to be able to express it in terms of the sum set, via (\ref{epm}), one needs to deal with the sum of $|A_d|^2$ in the right-hand side, rather than of $|A_d|^{\frac{3}{2}}.$

In the first formula of (\ref{th}) assume that $D'$ is a popular subset of $A-A$, namely
\begin{equation}\label{dprime}
D'=\left\{d\in A-A:\; |A_d|\geq \frac{1}{2}\frac{|A|^2}{|A-A|}\right\}.
\end{equation}
Then
$$
\sum_{d\in D'}  |A_d|^{\frac{3}{2}}\gg \left(\frac{|A|^2}{|A-A|}\right)^{\frac{1}{2}}\sum_{d\in D'}  |A_d|\gg  \left(\frac{|A|^2}{|A-A|}\right)^{\frac{1}{2}} |A|^2.
$$
In the second formula of (\ref{th}) replace $D'$ with $D^+$, defined by (\ref{dplus}). Then
\begin{equation}
\label{thth}
\begin{aligned}
 E(A,A-A) &\geq \frac{|A|^8 }{|A-A| E_3(A)}, \\
E(A,A+A)  &\geq  \frac{|A|^{10}}{|A+A|^2 E_3(A) \max_{d\in D^+}|A_d| }.
\end{aligned}
\end{equation}

\subsection{Some applications of Lemma \ref{twothr}} The estimates (\ref{thth}) enabled Schoen and Shkredov (\cite{SS}) to achieve progress on the sum set of a convex set problem. If $A=f([1,\ldots,N])$, where $f$ is a strictly convex real-valued function, then $|A-A|\gg
|A|^{\frac{5}{3}-o(1)}$, with a slightly worse estimate $|A+A|\gg |A|^{\frac{14}{9}-o(1)}$ for the sum-set. The reason for the two estimates being different is that dealing with the sum set necessitated bootstrapping  the earlier established exponent $\frac{3}{2}$, which had been obtained over the past ten years or so in various guises, with or without using the Szemer\'edi-Trotter theorem. (See e.g. \cite{ENR}, \cite{G}, \cite{IKRT}. The conjectured exponent in the convex set sum set problem is $2-o(1).$)

Li (\cite{Li}) -- see also his recent work with Roche-Newton (\cite{LR}) -- has adapted the approach of \cite{SS} to the sum-product problem, using a specific convex function, the exponential. This was also observed in another paper of Schoen and Shkredov (\cite{SS1}, Corollary
25). The result was the exponent $\frac{14}{11}-o(1)$. A closer look at this adaptation -- see the Appendix in this paper -- reveals that the exponential function has as much do with it as basically replacing $a$ with $\exp(\log a)$ in a variant of the old sum-product construction by Elekes (\cite{E}): applying the Schoen-Shkredov
trick, expressed in (\ref{th}) to the estimate (\ref{livar}) in the Appendix would improve the Elekes exponent $\frac{5}{4}$ to $\frac{14}{11}-o(1)$.

The same exponent $\frac{14}{11}-o(1)$ had been coincidentally obtained in Solymosi's work \cite{So1}, as stated in (\ref{wr1}) above.

This note uses the construction of \cite{So1} and combines it with an
estimate of the type (\ref{th}), thereby getting an improvement (\ref{res}) over (\ref{wr1}). The estimates involving the sum set end up being worse, because they require bootstrapping the estimates which themselves would yield the exponent $\frac{14}{11}-o(1)$ only. It will also use the Elekes construction in order to channel the estimates involving the ratio set into ones using the product set. These will have to bootstrap the Elekes exponent $\frac{5}{4}$, thus making the estimates involving the product set still worse.

\section{Proof of Theorem \ref{maint}}
The condition $|A|\geq 2$ is tantamount to assuming that $0\not\in A$. Consider a point set $A\times A$ in the coordinate plane.

The multiplicative energy $E_*(A)$ of $A$ is defined as the number of solutions of the following
equation
$$
E_*(A) = |\{ (a_1,\ldots a_4)\in A\times\ldots\times A:\; a_1/a_2 = a_3/a_4\}|.
$$
By the Cauchy-Schwarz inequality
\begin{equation}
E_*(A)\geq \max\left(\frac{|A|^4}{|A\cdot A|},\;\frac{|A|^4}{|A: A|}\right) .
\label{emult}\end{equation}
Geometrically,  $E_*(A)$ is the number of ordered pairs of points of $A\times A$ in the plane on straight lines through the origin, whose slopes $r$ are members of the ratio set $A:A$. Hence, elements of $A:A$ can be identified with lines through the origin supporting points of $A\times A$.

The proof will deal with an in some sense ``popular'' subset $L$ of these lines, $P$ denoting a subset of $A\times A$ supported on the lines in $L$. There will be two cases to consider.

\medskip
{\sf Ratio set case.} In order to establish the first two estimates of (\ref{res})
the notation $L$ will stand for the set of popular lines through the origin, namely those supporting at least $\frac{1}{2}|A|^2|A:A|$ points of $A\times A$. The subset of $A\times A$ supported on these lines $L$ will be denoted as $P$. One has $|P|\gg |A|^2$, as well as $|A|\leq |L|\leq |A:A|.$ Let us also use the notation $N$ for the maximum number of points per line, clearly $N\leq |A|.$
It follows from (\ref{emult}) that if $n(l)$ denotes the number of elements of $P$ supported on a line $l\in L$,
\begin{equation}
\sum_{l\in L} n^2(l)\gg \frac{|A|^4}{|A : A|}.
\label{csm}\end{equation}

\medskip
{\sf Product set case.} In order to establish the last two estimates of (\ref{res}), the same notations $P,L$ will apply to slightly differently defined sets, as well as the quantity $N$, as follows. The set $P$ will be a ``popular multiplicative energy'' subset of $A\times A$, constructed by the standard dyadic pigeonholing procedure by popularity, in terms of supporting points of $A\times A$, of all lines through the origin. The elements of $P$ are those points in $A\times A$ which lie on a set $L$ of straight lines passing through the origin, supporting between $N/2$ and $N$ points each, and such that
\begin{equation}
|L|N^2\gg \frac{|A|^4}{|A\cdot A|\log|A|}. \label{min}\end{equation}
Such sets $L,P$ always exist, by the pigeonhole principle and (\ref{emult}).

Important in the product set case will be the following bounds on $N$.

\begin{lemma}\label{ll} There exists $L,N$ satisfying (\ref{min}) and such that
\begin{equation}\label{nbd}
\frac{1}{2}\frac{|A|^2} {|A\cdot A|}\leq N\ll \frac{|A\pm A|^2|A\cdot A|}{|A|^3}.\end{equation}\end{lemma}
The lower bound in (\ref{nbd}) follows right away from (\ref{emult}) and a popularity argument. The upper bound comes from a variant of the Elekes construction (\cite{E}) apropos of sum-products. A variant of Lemma \ref{ll} can be found in the above-mentioned works  \cite{LR}, \cite{SS1}. A short proof is given in the Appendix.

\medskip
In both the ratio and product set cases, it can be assumed that $|L|\gg N$. It is clear in the ratio set case, where $N\leq |A|$ and thus, by (\ref{csm}),  $|L|\gg \frac{|A|^3}{|A:A|}\geq |A|.$

In the product set cases $|P|\approx |L|N$. Assume that $N\gg |L|.$
Since  $N$ satisfies (\ref{nbd}) and $LN^2$ is bounded from below by (\ref{min}), it follows that
$$|A\pm A|^6|A\cdot A|^4 \gg \frac{|A|^{13}}{\log |A|},$$
which is far better than the last two claims in (\ref{res}).

\medskip
Beginning now the proof proper, in either of the two cases above consider the sum set of $P$ with some other set $Q$ in the plane, with $|Q|\geq |P|$. (In the sequel $Q =\pm P$ or $P\pm P$). To obtain the vector sums, one translates the lines from $L$ to each point of $Q$, getting thereby some set ${\mathcal L}$ of lines with $|{\mathcal L}|\leq |L||Q|$.

The
Szemer\'edi-Trotter theorem, namely (\ref{STe}), enables one to estimate $|\mathcal L|$ from below:
$$
|L||Q|\ll |\mathcal L|^{\frac{2}{3}} |Q|^{\frac{2}{3}}.
$$
(It is easy to see that the fact that since $|L|\leq |Q|\leq |{\mathcal L}|\leq |L||Q|$, trivial terms in this application of (\ref{STe}) can be ignored.) Thus
\begin{equation}
|{\mathcal L}|\gg |L|^{\frac{3}{2}}|Q|^{\frac{1}{2}}. \label{lowerl}\end{equation} Let us call the number of points of $Q$ on a particular line $l\in \mathcal L$, the weight  $m(l)$ of $l$. The total weight $W$ of all lines in the collection $\mathcal L$ is by construction equal to $|L||Q|$.

The lines in ${\mathcal L}$ have been given weights, because the same line $l\in \mathcal L$ can contribute to the same vector sum in $P+Q$ at most $\max(N,m(l))$ times. Hence, let us lower the weights of lines, which are ``too heavy'': whenever $m(l)\geq N$, redefine it as $N$.

Therefore, for the total weight $W$ and the mean weight $\bar{m}$ per line, from (\ref{lowerl}) one has
\begin{equation}\label{weight}
W\leq |L||Q|,\qquad \bar m =\frac{W}{|\mathcal L|}\ll \sqrt{\frac{|Q|}{|L|}}.
\end{equation}

The Szemer\'edi-Trotter
theorem, namely (\ref{work}), tells one that the weight distribution over $\mathcal L$  obeys the inverse cube law. I.e.,
for $t\leq N$, one has
\begin{equation}
|\mathcal L_t| =|\{l\in \mathcal L:\, m(l)\geq t\}|\ll \frac{|Q|^2}{t^3}+\frac{|Q|}{t}\ll \frac{|Q|^2}{t^3},
\label{incube}\end{equation}
 as since $N\ll\sqrt{|Q|}$, the trivial term $\frac{|Q|}{t}$ gets dominated by the first term.

Now, let us look at the set ${\mathcal P}(\mathcal L)$ of all pair-wise intersections of lines from $\mathcal L$, and for an intersection point $p\in \mathcal P(\mathcal L)$ of some $k\geq 2$ lines
$l_1,\ldots, l_k$  look at the sum of the weights of the lines that intersect there:
$$
m(p) = \sum_{i=1}^k m(l_k).
$$
For any point set $\mathcal P\subseteq \mathcal P(\mathcal L)$, the number of weighted incidences between ${\mathcal P}$ and ${\mathcal L}$ is the sum over all pairs $(p,l)\in I(\mathcal P,\mathcal L)$, counting each pair $(p,l)$ with the weight $m(p)$.

The inverse cube weight distribution over the set of lines $\mathcal L$ enables one to use the Szemer\'edi-Trotter theorem rather efficiently for counting weighted incidences, similar to how it was done in the paper of Iosevich et al. (See \cite{IKRT}, Lemma 6).

\begin{lemma}\label{wst}
Suppose, $|Q|\geq |P|,$ and the weights of lines in ${\mathcal L}$ have been capped by $N$. For $x\in P+Q,$ let $n(x)$ be the number of realisations of $x$ as a sum.
Then for $t:\;N\ll t \leq |P|$,
\begin{equation}\label{use}
|\{x\in P+Q:\;n(x)\geq t\}|\ll \frac{|L|^{\frac{3}{2}} |Q|^{\frac{5}{2}}}{t^3},
\end{equation}
\end{lemma}

\begin{proof}
The condition $|Q|\geq |P|\gg N^2$ (which holds in both the ratio and product set cases) ensures that (\ref{incube}) is valid, since for all $l\in \mathcal L,\;m(l)\leq N$. Observe that for any point set ${\mathcal P}$, the number of weighted incidences of ${\mathcal L}$ with ${\mathcal P}$ can be bounded from above using dyadic decomposition of ${\mathcal L}$ by weight in excess of $\bar m$, via
\begin{equation}
\mathcal I \ll i_m(\mathcal P, \mathcal L_{\bar m}) + \sum_{j= 1}^{\log_2 N - \log_2 \bar m}\, (2^j \bar m) |I(\mathcal P,\mathcal L_{2^j \bar m})|,
\label{alt}\end{equation}
where $\mathcal L_{\bar m}$ denotes the set of lines from ${\mathcal L}$ whose weight is at most $\bar m$, $\mathcal L_{2^j \bar m}$ denotes the dyadic group of lines whose weights are approximately $2^j \bar m$, and the notation $i_m$ has been introduced in the context of Theorem \ref{STw}.

One uses the Szemer\'edi-Trotter theorem to estimate each number of incidences $|I(\mathcal P,\mathcal L_{2^j \bar m})|$, for $j>0$ involved and its weighted version for the first term. The condition (\ref{incube}) ensures that the  weighted Szemer\'edi-Trotter estimate (\ref{STew}) for the first term in (\ref{alt}) dominates as follows. For $j\geq 1$, one has
\begin{equation}\label{ests}
|I(\mathcal P,\mathcal L_{2^j \bar m})| \ll  (|\mathcal P| |\mathcal L_{2^j \bar m}|)^{\frac{2}{3}} + |\mathcal P| + |\mathcal L_{2^j \bar m}|.
\end{equation}
For the case $j=0$, Theorem \ref{STw} is used.
\begin{equation}\label{estss}
i_m(\mathcal P, \mathcal L_{\bar m}) \ll \bar m^{\frac{1}{3}}(|\mathcal P|W)^{\frac{2}{3}}+ \bar m |\mathcal P|+ W.
\end{equation}
Together with the rest of the bounds for each $|\mathcal L_{2^j \bar m}|$ coming from (\ref{incube}) this yields
\begin{equation}
\mathcal I \ll  \bar m^{\frac{1}{3}} (|\mathcal P|W)^{\frac{2}{3}} + N{|\mathcal P|} + W.
\label{int}\end{equation}
(See (\ref{tail}) below where the  bounds in (\ref{ests}) for the terms for $j>1$ are spelled out explicitly.)

Observe now  that
for $x\in P+ Q\subseteq \mathcal P(\mathcal L)$, with $n(x)$ denoting the number of realisations of a sum set element $x$, one always has $n(x)\leq m(x)$, where $m(x)$ is the weight of $x$ as a member of $\mathcal P(\mathcal L)$.

Applying (\ref{int}) to the point set
\begin{equation}
\mathcal P_t = \{x\in P+Q:\;n(x)\geq t\},\label{pt}\end{equation}
with the lower bound $t|\mathcal P_t|$ for ${\mathcal I}$,
one sees that for $t\gg N$ the term $N{|\mathcal P|}$ in the right-hand side of (\ref{int}) cannot possibly dominate the remaining terms.
Hence for $t\gg N$ one has
\begin{equation}\label{useprime}
|\mathcal P_t|\ll \frac{|L|^{\frac{3}{2}} |Q|^{\frac{5}{2}}}{t^3} + \frac{|L||Q|}{t},
\end{equation}
and has to be slightly more careful with the trivial term $\frac{|L||Q|}{t}$ in the bound (\ref{useprime}) and refine it, so that it becomes absorbed into the first term to get (\ref{use}).
This is clearly the case for $t\leq \sqrt[4]{|L||Q|^3}$, but not yet for higher $t$.

Let us now address the issue of large $t$.
The lines in ${\mathcal L}$ come in $|L|$ possible directions, and therefore no more than $|L|$ lines can be incident to a single point. Hence, lines from a dyadic group  $\mathcal L_{2^j \bar m}$ do not contribute to sets $\mathcal P_t$, whenever $t\gg |L|\cdot (2^j \bar m)$. This means that one needs, with (\ref{incube}) in view, only a ``tail'' estimate for the right-hand side of (\ref{alt}), with $t\approx |L| (2^j \bar m)$:
\begin{equation}
t|\mathcal P_t|  \ll \sum_{i=j}^{\log_2 N - \log_2 \bar m}\, (2^i \bar m) |I(\mathcal P_t,\mathcal L_{2^i \bar m})| \ll \frac{|Q|^\frac{4}{3}|\mathcal P_t|^{\frac{2}{3}}}{2^j \bar m } + N|\mathcal P_t| + \frac{ |Q|^2}{(2^j \bar m)^2}.
\label{tail}\end{equation}
It follows that, since $N\ll t\leq |P|$,
$$
|\mathcal P_t|  \ll \frac{|L|^{\frac{3}{2}}|Q|^{\frac{5}{2}}}{2^{3j}t^3}+  \frac{ |Q|^2}{(2^{j} \bar m)^2 t } \ll
\frac{|L|^{\frac{3}{2}}|Q|^{\frac{5}{2}}}{t^3} +  \frac{ |Q|^2|L|^2}{t^3 }.
$$
Since  $|Q|\geq |P|\geq |L|$, the first term in the last estimate dominates the second term, thus proving (\ref{use}).
\end{proof}

\medskip
Let us now apply the estimate (\ref{use}) with $Q=\pm P$; of these two the case $Q=-P$ will serve to estimate $E_3(P)$.

The assumptions of Lemma \ref{wst} are satisfied, and therefore
\begin{equation}\label{e3est}
E_3(P)=\sum_{x\in P-P} n^3(x)\ll  N^2|P|^2 + |L|^{\frac{3}{2}} |P|^{\frac{5}{2}} \sum_{j=1}^{2\log_2 |A|} \frac{2^{3j}}{2^{3j}} \ll
|L|^{\frac{3}{2}} |P|^{\frac{5}{2}} \log|A|,
\end{equation}
since in both the ratio and product set cases $N^2\ll \sqrt{|P||L|^3}.$

Using this, the first formula in (\ref{thth}), with $A$ replaced by $P$, yields:
\begin{equation}
E(P, P-P) \gg \frac{|P|^{\frac{11}{2}}}{ |L|^{\frac{3}{2}} |P-P| \log|A|}. \label{lowerbd1}
\end{equation}

To get a similar estimate for the sum set $P+P$, let us use the second formula in (\ref{thth}), with $A$ replaced by $P$ and $D=P-P$. To estimate the quantity $t=\max_{d\in D^+}|P_d|$ it suffices to observe that part of the energy $E(P)$ supported on $D^+$ is bounded from below by (\ref{needed}), where $A$ gets replaced by $P$, and from above, by Lemma \ref{wst}, as
$$O\left(N|P|^2+
\frac{ |L|^{\frac{3}{2}} |P|^{\frac{5}{2}}}{t}\right).$$

It follows that unless $|P+P|\gg \frac{|P|^2}{N}$, which is much better than (\ref{res}), one has
\begin{equation}\label{forplus}
\max_{d\in D^+}|P_d|\ll \frac{|L|^{\frac{3}{2}}|P+P|}{|P|^{\frac{3}{2}}}.
\end{equation}
Then, using the second estimate in (\ref{thth}) and (\ref{e3est}), one gets
\begin{equation}\label{lowerbd2}
E(P, P+P) \gg \frac{|P|^9}{|L|^3 |P+P|^3\log|A|}.
\end{equation}

On the other hand, one can use Lemma \ref{wst}  with $Q=P\pm P$ and estimate the quantity $E(P, P\pm P)$ from above. Then for any $t\gg N$:
\begin{equation}
E(P, P\pm P) \ll |P||P\pm P| t + \frac{|L|^{\frac{3}{2}} |P\pm P|^{\frac{5}{2}}}{t}, \label{upperbd}
\end{equation}
and choosing
$$t=\frac{|P\pm P|^{\frac{3}{4}} |L|^{\frac{3}{4}}}{\sqrt{|P|}}\gg N,$$ to match the two terms in (\ref{upperbd}) yields
\begin{equation}
E(P, P\pm P) \ll \sqrt{|P|}|P\pm P|^{\frac{7}{4}} |L|^{\frac{3}{4}}. \label{upperbddone}
\end{equation}

Combining this with (\ref{lowerbd1}) results in
\begin{equation}\label{ngood1}
|P-P|^{\frac{11}{4}} |L|^{\frac{9}{4}} \gg \frac{|P|^5}{\log |A|}.
\end{equation}
To obtain the first estimate of (\ref{res}) it suffices to note that one has (with what the notations $P,L$ stand for in the ratio set case) $|P|\gg |A|^2$, $|L|\leq |A:A|, $
plus $|P-P| \leq |A-A|^2$.

\medskip
To prove the second estimate in (\ref{res}) the estimates (\ref{upperbddone}) and (\ref{lowerbd2}) get put together.
It follows that
\begin{equation}\label{ngood2}
|P+P|^{\frac{19}{4}} |L|^{\frac{15}{4}} \gg \frac{|P|^{\frac{34}{4}} }{\log |A|}.
\end{equation}
Using $|L|\leq |A:A|$, $|P|\gg|A|^2$ and $|P+P|\leq |A+A|^2$ results in the second estimate in (\ref{res}).

\medskip
In the product set case, where  $|P| = |L|N$, the estimates (\ref{ngood1}, \ref{ngood2}) become

\begin{equation}\label{ngood3}\begin{aligned}
|A-A|^{\frac{11}{2}}  &\gg \frac{(|L|N^2)^{\frac{11}{4}}}{\sqrt{N} \log |A|},\\
|A+A|^{\frac{19}{2}} & \gg \frac{(|L|N^2)^{\frac{19}{4}}}{N \log |A|}.\end{aligned}
\end{equation}
Lemma \ref{ll} now supplies a non-trivial upper bound on $N$.
Substituting the bounds from (\ref{min}) and (\ref{nbd}) into (\ref{ngood3})  then yields the last two estimates of (\ref{res}) and completes the proof of Theorem \ref{maint}. \qed

\subsection{Appendix. Proof of Lemma \ref{ll}}
The lower bound of (\ref{nbd}) is merely the multiplicative version of the popularity argument behind (\ref{dplus}), resulting in (\ref{needed}). Indeed, the multiplicative energy of $A$ supported on those lines through the origin, which correspond to ratios $r\in A:A$, whose number of realisations  $n(r)\geq\frac{1}{2}\frac{|A|^2}{|A\cdot A|},$ is $\gg \frac{|A|^4}{|A\cdot A|}$.

On the other hand, it was proven in \cite{LR}, \cite{SS1} that the multiplicative energy coming from the lines through the origin, supporting at least $N$ points  of $A\times A$ (that from the set of all ratios $r\in A:A$ such that $n(r)\geq N$), is $O\left(\frac{|A\pm A|^2|A|}{N}\right)$.
This, together with (\ref{min}) settles the upper bound in (\ref{nbd}) and proves Lemma \ref{ll}.

\medskip
For completeness sake, a simple version of the proof of the upper bound in (\ref{nbd}) for $N$ is given below. This bound was derived in \cite{LR}, \cite{SS1} via the Szemere\'ed-Trotter type estimates for convex functions, using a particular example of a convex function, the exponential. Let us show that the bound in question, in fact,  represents a variant of the well known construction of Elekes (\cite{E}) apropos of sum-products, which gave the exponent $\frac{5}{4}$ implicit in the bounds (\ref{nbd}). The notation in the forthcoming argument is somewhat independent from the rest of the paper.

Consider a set $A$, not containing zero and a set of lines $L = \{y = \frac{d + x}{a}\}$, where $d$ is an element of the difference set $A-A$ (or the sum set $A+A$, the modification required being trivial) and $a\in A$. Clearly there are $|A-A||A|$ lines. Therefore, the number of points in a set $P_t$, where at least $t\leq |A|$ lines intersect is, by
 (\ref{work}), bounded as
 $$
 |P_t|\ll \frac{|A-A|^2|A|^2}{t^3}.
 $$
 Given a point $(x,y)\in P_t$, one has an intersections of at least $t$ lines of $L$ there, namely $y=\frac{d_i + x}{a_i}$ for at least $t$ different pairs $(d_i,a_i)$.
 For each $d_i$ take some fixed representation $d_i = u_i-v_i$, let $a\in A$ be a variable. Clearly
 $d_i + x  = (u_i-a) + (a-v_i+x) = d_i'+x'$. I.e., $P_t$ is such that if it contains one point $(x,y)$, then it contains at least $|A|$ distinct points with the same ordinate.

 Conversely, if $R_t$ is the subset of ratios from $A:A$ which have at least $t$ realisations, it is clearly a subset of the set of ordinates of the points from $P_t$. Indeed, if
 $r=\frac{a'_1}{a_1}= \ldots =\frac{a'_k}{a_k}$, then it equals to the ordinate of a horizontal family of intersection points of at least $k$ lines, identified by the pairs $(d_i = a_i' - a,a_i)$, these intersection points having the abscissae $x=a$.

 It follows that
\begin{equation}\label{livar}
 |R_t|\ll \frac{|A-A|^2|A|}{t^3},
 \end{equation}
 and hence the multiplicative energy supported on $R_t$ is ${\displaystyle O\left(\frac{|A-A|^2|A|}{t}\right).}$ Comparing this with the lower bound in terms of $\frac{|A|^4}{|A\cdot A|}$ gives the upper bound in (\ref{nbd}), where $t$ has been replaced by $N$. \qed

\section{Acknowledgement} The author thanks Oliver Roche-Newton for pointing out that the estimates (\ref{ngood3}) needed a non-trivial upper bound for $N$ in the product set case.


\vspace{2cm}
\begin{thebibliography}{4}

\bibitem{E} G. Elekes. {\em On the number of sums and products}. Acta Arithmetics {\bf 81} (1997), 365--367.

\bibitem{ENR} G. Elekes, M. Nathanson, I. Z. Ruzsa. {\em  Convexity and sumsets.} J. Number Theory {\bf 83} (2000), 194–-201.

\bibitem{ES} P. Erd\H os, E. Szemer\'edi. {\em On sums and products of integers.} Studies in Pure Math. (Birkh\"auser, Basel, 1983) 213--218.

\bibitem{G} M.Z. Garaev. {\em On an additive representation associated with the $L^1$-norm of an exponential sum.} Rocky Mountain J. Math. {\bf 37} (2007), no. 5, 1551–-1556.

\bibitem{IKRT} A. Iosevich, S. Konyagin, M.  Rudnev, V Ten. {\em Combinatorial complexity of convex sequences.} Discrete Comput. Geom. {\bf 35} (2006), no. 1, 143–-158.

\bibitem{KK} N.H. Katz, P. Koester. {\em On additive doubling and energy.} SIAM J. Discrete Math., 24(4) (2010), 1684–-1693.

\bibitem{Li} L. Li.  {\em On a theorem of Schoen and Shkredov on sumsets of convex sets.} Preprint {\sf arXiv math:11108.4382} (2011), 6pp.

\bibitem{LR} L. Li, O. Roche-Newton. {\em Convexity and a sum-product type estimate.} Preprint 2011, 9pp.

\bibitem{SS0} T. Schoen, I. Shkredov. {\em  On a question of Cochrane and Pinner concerning multiplicative subgroups.} Preprint {\sf arXiv math: 1008.0723} (2010), 10pp.

\bibitem{SS} T. Schoen, I. Shkredov. {\em On sumsets of convex sets.} Preprint {\sf arXiv math:1105.3542} (2011), 6pp.

\bibitem{SS1} T. Schoen, I. Shkredov.  {\em Higher moments of convolutions.}  Preprint {\sf arXiv math:1110.2986} (2011), 36pp.

\bibitem{SV} I. Shkredov, I. Vyugin. {\em  On additive shifts of multiplicative subgroups.} Preprint {\sf arXiv math: 1102.1172} (2011), 18pp.

\bibitem{So} J. Solymosi. {\em On the number of sums and products} Bull. London Math. Soc. {\bf 37} (2005), no. 4, 491–-494.

 \bibitem{So1} J. Solymosi.   {\em Bounding multiplicative energy by the sumset.} Adv. Math. {\bf 222} (2009), no. 2, 402–-408.

 \bibitem{SoT} J. Solymosi, T. Tao. {\em An incidence theorem in higher dimensions.} Preprint {\sf arXiv math: 1103.2926} (2011), 24pp.

\bibitem{ST} E. Szemer\'edi, W.T. Trotter, Jr. {\em Extremal problems in discrete geometry.} Combinatorica {\bf 3} (1983) 381--392.

 \bibitem{T}  C. D. T\'oth. {\em   The Szemer\'edi-Trotter Theorem in the Complex Plane.}
 Preprint {\sf arXiv math/0305283} (2003), 23pp.


\end{thebibliography}
\end{document}